\documentclass[11pt]{article}

\usepackage{amsmath}
\usepackage{amsthm}
\usepackage{amsfonts}
\usepackage{setspace}
\usepackage{fullpage}
\usepackage{amssymb}
\usepackage{enumitem}
\usepackage{tikz}
\usetikzlibrary{arrows}
\usepackage{ifpdf}
\usepackage{hyperref}

\newtheorem{theorem}{Theorem} 
\newtheorem*{theorem*}{Theorem} 
\newtheorem*{corollary*}{Corollary} 
\newtheorem*{lemma*}{Lemma}
\newtheorem{lemma}[theorem]{Lemma}
\newtheorem{corollary}[theorem]{Corollary} 
\newtheorem{definition}[theorem]{Definition}

\newtheorem{claim}[theorem]{Claim}

\newtheorem{proposition}[theorem]{Proposition} 
\newtheorem*{proposition*}{Proposition} 
\theoremstyle{remark}

\newcommand{\A}{\mathcal{A}}
 
\newcommand{\C}{\mathcal{C}}

\newcommand{\F}{\mathcal{F}}
\renewcommand{\P}{\mathcal{P}}
\newcommand{\h}{\mathcal{H}}
\newcommand{\I}{\mathcal{I}}

\newcommand{\Ordo}{\mathcal{O}}

\newcommand{\abs}[1]{\left\lvert{#1}\right\rvert}
\newcommand{\floor}[1]{\left\lfloor{#1}\right\rfloor}
\newcommand{\ceil}[1]{\left\lceil{#1}\right\rceil}
\newcommand{\lubell}{\sum_{A\in\A} \frac{1}{\binom{n}{\abs{A}}}}
\DeclareMathOperator{\La}{La}

\begin{document}
	
\title{An improvement of the general bound on the largest family of subsets avoiding a subposet}

\author{D\'aniel Gr\'osz \footnote{Department of Mathematics, University of Pisa.  e-mail: \href{mailto:groszdani@gmail.com}{groszdanielpub@gmail.com}} \and Abhishek Methuku \footnote{ Department of Mathematics, Central European University, Budapest. e-mail: \href{mailto:abhishekmethuku@gmail.com}{abhishekmethuku@gmail.com}} \and Casey Tompkins \footnote{ Department of Mathematics, Central European University, Budapest. e-mail: \href{mailto:ctompkins496@gmail.com}{ctompkins496@gmail.com}}}

\maketitle
	
	\begin{abstract}
		Let $\La(n,P)$ be the maximum size of a family of subsets of $[n]= \{1,2, \ldots, n \}$ not containing $P$ as a (weak) subposet, and let $h(P)$ be the length of a longest chain in $P$. The best known upper bound for $\La(n,P)$ in terms of $\abs{P}$ and $h(P)$ is due to Chen and Li, who showed that $\La(n,P) \le \frac{1}{m+1} \left(\abs{P} + \frac{1}{2}(m^2 +3m-2)(h(P)-1) -1 \right) {\binom {n} {\lfloor n/2 \rfloor}}$ for any fixed $m \ge 1$. 
		
		In this paper we show that $\La(n,P) \le  \frac{1}{2^{k-1}} \left(\abs P + (3k-5)2^{k-2}(h(P)-1) - 1 \right) {n \choose \left\lfloor n/2\right\rfloor }$ for any fixed $k \ge 2$, improving the best known upper bound. By choosing $k$ appropriately, we obtain that $\La(n,P) = \Ordo\left( h(P) \log_2\left(\frac{\abs{P}}{h(P)}+2\right) \right) {n \choose \floor{n/2} }$ as a corollary, which we show is best possible for general $P$. We also give a different proof of this corollary by using bounds for generalized diamonds.  We also show that the Lubell function of a family of subsets of $[n]$ not containing $P$ as an induced subposet is $\Ordo(n^c)$ for every $c>\frac{1}{2}$.
		
	\end{abstract}
	
	\section{Introduction}
	Let $[n] = \{1,2,\ldots,n\}$ and $2^{[n]}$ be the power set of $[n]$.   For two partially ordered sets (posets), $P$ and $Q$, $P$ is said to be a subposet of $Q$ if there exists an injection $\varphi$ from $P$ into $Q$ so that $x\le y$ in $P$ implies $\varphi(x) \le \varphi(y)$ in $Q$, whereas $P$ is said to be an induced subposet of $Q$ if there exists an injection  $\varphi'$ from $P$ into $Q$ such that $x\le y$ in $P$ if and only if $\varphi'(x) \le \varphi'(y)$ in $Q$. Every family of sets $\A \subset 2^{[n]}$ may be viewed as a poset with respect to inclusion.  
	
	Define $\La(n,P) = \max\{\abs{\A}: \A \subset 2^{[n]} \mbox{ and } P  \mbox{ is not a subposet of } \A\}$ and let $\La^{\#}(n,P) = \max\{\abs{\A}: \A \subset 2^{[n]} \mbox{ and } P  \mbox{ is not an induced subposet of } \A\}$.  The function $\La(n,P)$ has been studied extensively. Such results are all extensions of a famous theorem of Sperner \cite{sperner1928satz} asserting that the size of the largest antichain in $2^{[n]}$ (containment-free family) is $\binom{n}{\floor{n/2}}$.  Erd\H{o}s \cite{erdos1945lemma} extended this result to $k+1$-paths $P_{k+1}$.  A central open problem in this area is to determine the value of $\La(n,D_2)$, where $D_2$, the diamond poset, is defined by $4$ elements $w,x,y,z$ with $w \le x,y \le z$. Posets for which $\La(n,P)$ has been studied include crowns \cite{lu2014crown}, harps \cite{griggs2012diamond}, generalized diamonds \cite{griggs2012diamond}, the butterfly poset \cite{de2005largest}, fans \cite{griggs2012poset}, $V$'s and $\Lambda$'s \cite{katona1983extremal}, the $N$ poset \cite{griggs2008}, forks \cite{krfork}, and recently the complete $3$ level poset $K_{r,s,t}$ \cite{patkos} among many others.  
	
	In another direction,  it is interesting to determine general bounds on $\La(n,P)$ depending on the size of $P$ and the length of the largest chain in $P$, denoted $h(P)$.   
	
	If $T$ is a tree then Bukh \cite{bukh2009set} proved $\La(n,T) \le (h(T)-1) \binom{n}{\floor{n/2}} (1 + \Ordo(\frac1n))$, thereby establishing the asymptotically optimal bound for the case of trees. The first upper bound of $\La(n,P)$ for general posets $P$ in terms of $\abs{P}$ and $h(P)$ is due to Burcsi and Nagy \cite{burcsi2013method}.
	
	\begin{theorem} [Burcsi, Nagy \cite{burcsi2013method}]
		\label{doublechain}
		For any poset $P$, when $n$ is sufficiently large, we have
		\begin{equation}
			\La(n,P) \le \left( \frac{\abs{P}+h(P)}{2} -1 \right) \binom{n}{\floor{\frac{n}{2}}}.
		\end{equation}
	\end{theorem}
	
	In their paper, they introduced a generalization of the chain, called a double chain, and used a Lubell-style double counting argument to deduce the bound.    This object was generalized by Chen and Li \cite{chen2014note} who improved their upper bound to the following:
	
	\begin{theorem} [Chen, Li \cite{chen2014note}]
		\label{chenli}
		For any poset $P$, when $n$ is sufficiently large, the inequality
		\begin{equation}
			\label{eq:genChenLi}
			\La(n,P) \le \frac{1}{m+1} \left(\abs{P} + \frac{1}{2}(m^2 +3m-2)(h(P)-1) -1 \right) {\binom {n} {\floor{\frac{n}{2}}}}
		\end{equation}
		holds for any fixed $m \ge 1$.
	\end{theorem}
	
	Putting $m = \ceil{\sqrt{ \frac{\abs{P}}{h(P)}} } $ in the above formula, they obtained 
	\begin{equation}
		\label{eq:OboundLi}
		\La(n,P) = \Ordo(\abs{P}^{1/2}h(P)^{1/2}) \binom{n}{\floor{\frac{n}{2}}}.
	\end{equation}
	
	We further improve Theorem \ref{chenli}, by showing that
	
	\begin{theorem}\label{mainthm}
		For any poset $P$, when $n$ is sufficiently large, the inequality
		\begin{displaymath}
			\La(n,P) \le  \frac{1}{2^{k-1}} \left(\abs P + (3k-5)2^{k-2}(h(P)-1) - 1 \right) {n \choose \floor{\frac{n}{2}} }
		\end{displaymath}
		holds for any fixed $k\geq 2$.
	\end{theorem}
	
	Notice that putting $k = 2$, we get Theorem \ref{doublechain} and Theorem \ref{chenli} for $m$ = 1. Putting $k =  3$, we get Theorem \ref{chenli} for $m = 3$. For $k > 3$, our result strictly improves Theorem \ref{chenli}.
	
	By choosing $k$ appropriately in our theorem, we obtain the following improvement of \eqref{eq:OboundLi}:
	
	\begin{corollary}\label{maincor}
		For every poset $P$ and sufficiently large $n$,
		\begin{displaymath}
			\La(n,P) = \Ordo\left( h(P) \log_2\left(\frac{\abs{P}}{h(P)}+2\right) \right) {n \choose \left\lfloor \frac{n}{2}\right\rfloor }.
		\end{displaymath}
	\end{corollary} 
	The following proposition shows that this bound cannot be improved for general $P$.
	\begin{proposition} \label{lowerbound}
		For $P=K_{a,a,\ldots,a}$, we have
		\[ \La(n,P) \ge \left( (h(P)-2)\log_2 a \right) {n \choose \left\lfloor \frac{n}{2}\right\rfloor } =  \left( (h(P)-2)\log_2 \left (\frac{\abs{P}}{h(P)} \right ) \right) {n \choose \left\lfloor \frac{n}{2}\right\rfloor }. \]
	\end{proposition}

	It is interesting to note that much less is known about the induced version. The only known general bound on $\La^{\#}(n,P)$ has a much weaker constant than for the non-induced problem due to its dependence on the constant term of the higher dimensional variant of the Marcus-Tardos theorem \cite{marcus2004,klazar2007}.
	
	\begin{theorem} [Methuku, P\'alv\"olgyi \cite{methukup}] \label{inducedLa}
		For every poset $P$, there is a constant $C$ such that the size of any family of subsets of $[n]$ that does not contain an induced copy of $P$ is at most $C\binom{n}{\left\lfloor \frac{n}{2}\right\rfloor}$.
	\end{theorem}
	
	Define the Lubell function of a family of subsets of $[n]$ as $l_n(\A)=\lubell$. The Lubell function is the sum of the proportion of sets selected of each size; clearly $l_n(\A) \geq \frac{\abs \A}{\binom{n}{\left\lfloor \frac{n}{2}\right\rfloor }}$. Define $\lambda^\#_n(P)$ as the maximum value of $l_n(\A)$ over all induced $P$-free families $\A \subset 2^{[n]}$. While $\frac{\La^\#(n,P)}{\binom{n}{\left\lfloor \frac{n}{2}\right\rfloor }}$ is known to have a constant bound for every $P$, it is not currently known if $\lambda^\#_n(P)$ also has a constant bound for every $P$. We prove the following result about $\lambda^\#_n(P)$.
	
	\begin{theorem}\label{induced}
		For every poset $P$ and every $c>\frac 1 2$,
		\[\lambda^\#_n(P)=\Ordo(n^c).\]
	\end{theorem}
	
	\medskip
	
	The paper is organized as follows: in the second section we define our more general chain structure called an \emph{interval chain} and give a proof of Theorem \ref{mainthm} and Corollary \ref{maincor} using it.  In the third section we give another proof of Corollary \ref{maincor}, with a better constant, using an embedding of arbitrary posets into a product of generalized diamonds.  We also give a proof of Proposition \ref{lowerbound}.  In the fourth section we use the interval chain technique to prove Theorem \ref{induced}.

	\section{Interval chains and the proof of Theorem \ref{mainthm}}
	\label{intchainsec}
	We begin by proving some lemmas which allow us to extend Lubell's argument to more general structures.    Let $\pi \in S_n$ be a permutation and $A \subset [n]$ be a set, then $A^\pi$ denotes the set $\{\pi(a):a\in A\}$.   Moreover, for a collection of sets $\h \subset 2^{[n]}$ we define $\h^\pi$ to be the collection $\{A^\pi:A \in \h\}$.
	\begin{lemma}
		\label{l1}
		Let $\h \subset 2^{[n]}$ be a collection of sets and $A \subset [n]$ be any set.  Let $N_i = N_i(\h)$ be the number of sets in $\h$ of cardinality $i$.  The number of permutations $\pi$ such that $A \in \h^{\pi}$ is $N_{\abs{A}} \abs{A}! (n-\abs{A})!$.
	\end{lemma}
	\begin{proof}
		Let $S_1,\ldots, S_{N_{\abs{A}}}$ be the collection of sets in $\h$ of size $\abs{A}$.  The number of permutations $\pi$ such that $S_i$ is mapped to $A$ is $\abs{A}!(n-\abs{A})!$, since we can map the elements of $S_i$ to $A$ arbitrarily and the elements of $[n] \setminus S_i$ to $[n] \setminus A$ arbitrarily.   Moreover, no permutation $\pi$ maps two sets, $S_i,S_j$, to $A$, for then $S_i^\pi =S_j^\pi$, that is $\{\pi(s):s \in S_i\} = \{\pi(s):s \in S_j\}$ and so $S_i = S_j$, a contradiction. Since there are $N_{\abs{A}}$ sets in $\h$ of size $\abs{A}$, and we have shown that the set of permutations mapping each of them to $A$ is disjoint.  It follows  that the number of permutations $\pi$ such that $A \in \h^\pi$ is $N_{\abs{A}} \abs{A}! (n-\abs{A})!$.
	\end{proof}
	
	For a collection $\h \subset 2^{[n]}$ and a poset, $P$, let $\alpha(\h,P)$ denote the size of the largest subcollection of $\h$ containing no $P$.  Observe that $\alpha(\h,P) = \alpha(\h^\pi,P)$ for all $\pi \in S_n$ since containment relations are unchanged by permutations of $[n]$.
	\begin{lemma}\label{bound}
		Let $\A$ be a $P$-free family in $2^{[n]}$ and $\h$ be a fixed collection.   We have
		\begin{displaymath}
			\sum_{A \in \A} \frac{N_{\abs{A}}}{\binom{n}{\abs{A}}} \le \alpha(\h,P).
		\end{displaymath}
		In particular, if all of the $N_i$ are equal to the same number $N$, we have
		\begin{displaymath}
			\sum_{A \in \A} \frac{1}{\binom{n}{\abs{A}}} \le \frac{\alpha(\h,P)}{N}.
		\end{displaymath}
	\end{lemma}
	\begin{proof}
		We will double count pairs $(A,\pi)$ where $A \in \h^\pi$.  First fix a set $A$, then Lemma \ref{l1} shows there are $N_{\abs{A}} \abs{A}! (n- \abs{A})!$ permutations for which $A \in \h^\pi$.  Now fix a permutation $\pi \in S_n$.  By the definition of $\alpha(\h,P)$ we have $\abs{\A \cap \h^\pi} \le \alpha(\h,P)$.  Since there are $n!$ permutations, it follows that the number of pairs $(A,\pi)$ is at most $\alpha(\h,P) n!$.  Thus, we have
		\begin{displaymath}
			\sum_{A \in \A} N_{\abs{A}} \abs{A}! (n- \abs{A})! \le \alpha(\h,P) n!,
		\end{displaymath}
		and rearranging yields the result.
	\end{proof}
	
	We introduce a structure $\h \subset 2^{[n]}$ which we call a \emph{$k$-interval chain}.  Define the interval $[A,B]$ to be the set $\{C:A \subseteq C \subseteq B\}$.  Fix a maximal chain $\C=\{A_0=\emptyset, A_1, \ldots, A_{n-1}, A_n=[n]\}$ where $A_i \subset A_{i+1}$ for $0\le i \le n-1$. From $\C$ we define the $k$-interval chain $\C_k$ as  
	\begin{displaymath}
		\C_k = \bigcup_{i=0}^{n-k} [A_i, A_{i+k}].
	\end{displaymath}

	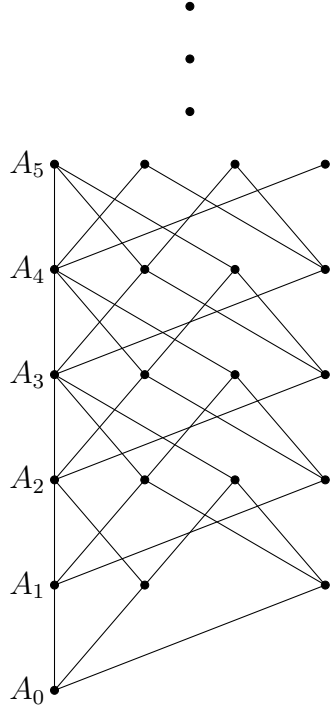
\begin{figure}[h]
		\centering
		\begin{tikzpicture}[line cap=round,line join=round,>=triangle 45,x=1.2cm,y=1.4cm]
		\draw (4.0,2.0)-- (4.0,3.0);
		\draw (4.0,3.0)-- (4.0,4.0);
		\draw (4.0,4.0)-- (4.0,5.0);
		\draw (4.0,5.0)-- (4.0,6.0);
		\draw (4.0,6.0)-- (4.0,7.0);
		\draw (4.0,2.0)-- (5.0,3.0);
		\draw (4.0,2.0)-- (7.0,3.0);
		\draw (4.0,3.0)-- (5.0,4.0);
		\draw (5.0,3.0)-- (4.0,4.0);
		\draw (5.0,3.0)-- (6.0,4.0);
		\draw (6.0,4.0)-- (7.0,3.0);
		\draw (4.0,3.0)-- (7.0,4.0);
		\draw (4.0,4.0)-- (5.0,5.0);
		\draw (5.0,4.0)-- (4.0,5.0);
		\draw (5.0,4.0)-- (6.0,5.0);
		\draw (4.0,5.0)-- (6.0,4.0);
		\draw (6.0,5.0)-- (7.0,4.0);
		\draw (4.0,4.0)-- (7.0,5.0);
		\draw (4.0,5.0)-- (5.0,6.0);
		\draw (4.0,6.0)-- (5.0,5.0);
		\draw (4.0,6.0)-- (6.0,5.0);
		\draw (5.0,5.0)-- (6.0,6.0);
		\draw (6.0,6.0)-- (7.0,5.0);
		\draw (4.0,5.0)-- (7.0,6.0);
		\draw (4.0,7.0)-- (5.0,6.0);
		\draw (4.0,7.0)-- (6.0,6.0);
		\draw (4.0,6.0)-- (5.0,7.0);
		\draw (4.0,6.0)-- (7.0,7.0);
		\draw (5.0,6.0)-- (6.0,7.0);
		\draw (6.0,7.0)-- (7.0,6.0);
		\draw (5.0,4.0)-- (7.0,3.0);
		\draw (5.0,5.0)-- (7.0,4.0);
		\draw (5.0,6.0)-- (7.0,5.0);
		\draw (5.0,7.0)-- (7.0,6.0);
		\begin{scriptsize}
		\draw [fill=black] (4.0,2.0) circle (1.5pt);
		\draw[color=black] (3.7, 2.0) node {$\mathbf{\mbox{\large $A_0$}}$};
		\draw [fill=black] (4.0,3.0) circle (1.5pt);
		\draw[color=black] (3.7, 3.0) node {$\mathbf{\mbox{\large $A_1$}}$};
		\draw [fill=black] (4.0,4.0) circle (1.5pt);
		\draw[color=black] (3.7, 4.0) node {$\mathbf{\mbox{\large $A_2$}}$};
		\draw [fill=black] (4.0,5.0) circle (1.5pt);
		\draw[color=black] (3.7, 5.0) node {$\mathbf{\mbox{\large $A_3$}}$};
		\draw [fill=black] (4.0,6.0) circle (1.5pt);
		\draw[color=black] (3.7, 6.0) node {$\mathbf{\mbox{\large $A_4$}}$};
		\draw [fill=black] (4.0,7.0) circle (1.5pt);
		\draw[color=black] (3.7, 7.0) node {$\mathbf{\mbox{\large $A_5$}}$};
		\draw [fill=black] (5.0,3.0) circle (1.5pt);
		\draw [fill=black] (7.0,3.0) circle (1.5pt);
		\draw [fill=black] (5.0,4.0) circle (1.5pt);
		\draw [fill=black] (6.0,4.0) circle (1.5pt);
		\draw [fill=black] (7.0,4.0) circle (1.5pt);
		\draw [fill=black] (5.0,5.0) circle (1.5pt);
		\draw [fill=black] (6.0,5.0) circle (1.5pt);
		\draw [fill=black] (7.0,5.0) circle (1.5pt);
		\draw [fill=black] (5.0,6.0) circle (1.5pt);
		\draw [fill=black] (6.0,6.0) circle (1.5pt);
		\draw [fill=black] (7.0,6.0) circle (1.5pt);
		\draw [fill=black] (5.0,7.0) circle (1.5pt);
		\draw [fill=black] (7.0,7.0) circle (1.5pt);
		\draw [fill=black] (6.0,7.0) circle (1.5pt);
		\draw [fill=black] (5.5,8.0) circle (1.5pt);
		\draw [fill=black] (5.5,8.5) circle (1.5pt);
		\draw [fill=black] (5.5,7.5) circle (1.5pt);
		\end{scriptsize}
		\end{tikzpicture}
		\caption{3-interval chain}
		\label{fig:IntervalChain}
	\end{figure}
	
	See Figure~\ref{fig:IntervalChain} for an example of an interval chain. We begin by deriving some properties of interval chains. In the rest of the paper we shall work with the $k$-interval chain $\C_k^0$ defined by $A_i=[i]$; other $k$-interval chains are related to it by permutation. It is easy to see that the indicator vectors of the sets in $\C_k^0$ consist of an initial segment of 1's, then $k$ arbitrary bits, followed by 0's. We call the number of 1's in a 0--1 vector the weight of the vector (which is the size of the corresponding set).
	
	We will now prove a sequence of lemmas that we use to bound the number of sets in a $P$-free subfamily of a $k$-interval chain. We call two sets related if one of them contains the other.   The idea, following Burcsi, Nagy \cite{burcsi2013method} and Chen, Li \cite{chen2014note}, is to partition $P$ into $h(P)$ antichains and embed the antichains into a given subcollection of $\C_k^0$, one by one,  in such a way that every set in one antichain is related  to every set in the next antichain.   To this end, we ignore those sets in $\C_k^0$ which may be unrelated to some previously embedded set.   The key lemma, Lemma \ref{unrelated}, gives an upper bound to how many sets we must ignore.
	
	For convenience, from now on we identify sets and their indicator vectors.
	
	\begin{lemma}\label{levelsize}
		For $k\leq m \leq n-k$, the number of sets of size $m$ in a $k$-interval chain is $2^{k-1}$. The number of such sets which have at least $j$ $0$'s before the last 1 is $\sum_{h=j}^{k-1} \binom{k-1}{h}$.
	\end{lemma}
	\begin{proof}
		We give a bijection $\varphi$ between 0--1 vectors of length $k-1$ and sets of size $m$ in $\C_k^0$. Let $u$ be a 0--1 vector of length $k-1$, and let $w$ be the weight of $u$. Let $\varphi(u)=\overbrace{111\ldots1}^{m-w-1}\overbrace{\vphantom1 u}^{k-1}1\overbrace{0000000\ldots0}^{n-m-k+w+1}$. A set of size $m$ in $\C_k^0$ is assigned to $u$ if and only if in its indicator vector the last $k-1$ bits leading up to (but not including) the last 1 coincide with $u$.   We show $\varphi$ is injective and surjective.   If $\varphi(u) = \varphi(v)$, then both $u$ and $v$ consist of the $k-1$ bits preceding the final 1 so $u=v$, and it follows $\varphi$ is injective.  Now, take an arbitrary weight $m$ vector, $z$, corresponding to a set in $\C_k^0$.  Find the last $1$ occurring in $z$ and let $u$ be the vector of length $k-1$ immediately preceding it (such a vector exists since $m \ge k$).  Then $\varphi(u) = z$, and we have that $\varphi$ is surjective.
		
		There are $2^{k-1}$ vectors $u$ of length $k-1$. Among such vectors,  $\sum_{h=j}^{k-1} \binom{k-1}{h}$ of them have at least $j$ 0's, and precisely these vectors are the ones mapped to vectors with at least $j$ 0's before the last 1. The condition $k\leq m \leq n-k$ guarantees that both $m-w-1$ and $m+k-w+1$ are between 0 and $n$.
	\end{proof}
	
	\begin{lemma}\label{unrelated}
		For $3k-3\leq m \leq n-k+1$, the number of sets in a $k$-interval chain which have  size at most $m-1$, and which are unrelated to some other set in the $k$-interval chain of size at least $m$, is $(3k-5)2^{k-2}$.
	\end{lemma}
	
	\begin{proof}
		We will show that the sets in the $k$-interval chain $\C_k^0$, which are unrelated to at least one set of size $m$ or greater in $\C_k^0$ are: all indicator vectors in $\C_k^0$ of weight between $m-1$ and $m-(k-2)$ inclusive; plus, among indicator vectors with weight $m-i$ with $k-1\leq i \leq 2k-3$, those which have at least $i-k+2$ 0's before the last 1. Let's denote the collection of these vectors by $\mathcal{S}$. Then, by Lemma \ref{levelsize}, we can calculate the number $\abs{\mathcal{S}}$ of such vectors:
		\begin{multline*}
			(k-2)2^{k-1}+\sum_{i=k-1}^{2k-3}\sum_{h=i-k+2}^{k-1}{k-1 \choose h}=(k-2)2^{k-1}+\sum_{j=1}^{k-1}\sum_{h=j}^{k-1}{k-1 \choose h}=\\
			=(k-2)2^{k-1}+\sum_{h=1}^{k-1}h{k-1 \choose h}=(k-2)2^{k-1}+(k-1)2^{k-2}=(3k-5)2^{k-2}.
		\end{multline*}
		
		First we show that if $v \in \mathcal{S}$, there is a vector of weight $m$ in $\C_k^0$ which is unrelated to it. Let $m-i$ be the weight of $v$. We need to change at least one 1 to 0 (i.e., remove some elements), and change $i$ more 0's to 1's than we just removed (that is, add $i$ more elements than we just removed).
		
		Assume that the last 1 in $v$ is at index $l$, so the first $l-k$ elements in $v$ are 1's. Also assume that there are $j$ 0's in $v$ with an index less than $l$. We can change the $l^{th}$ entry of $v$ from 1 to 0, and change the first $i+1$ 0's in $v$ to 1's because $i+1\leq j+k-1$. We obtain either a vector with at least $l-k+2$ initial 1's, and 0's from an index  $\le l$; or a vector with $l-1$ initial 1's, and 0's from an index $\le l+k-1$ (see the figure below). Either way the difference between the index of the last 1 and the first 0 is at most $k-1$, so the obtained vector is in $\C_k^0$.
		\noindent \begin{center}
			\begin{tabular}{ccc}
				$\begin{gathered}\overbrace{111111111}^{\textrm{initial segment}}\overbrace{00010}^{\leq k-1}1\overbrace{00000}^{k-1}000\\
				\downarrow\\
				\overbrace{111111111}\overbrace{11010}^{\leq k-1}0\overbrace{00000}^{k-1}000
				\end{gathered}
				$ & or & $\begin{gathered}\overbrace{111111111}^{\textrm{initial segment}}\overbrace{00010}^{\leq k-1}1\overbrace{00000}^{k-1}000\\
				\downarrow\\
				\overbrace{111111111}\overbrace{11111}^{\leq k-1}0\overbrace{11000}^{k-1}000
				\end{gathered}
				$\tabularnewline
			\end{tabular}
		\end{center}
		
		Conversely, we prove that if $v$ (which is of weight at most $m-i$, $i\geq1$) is not in $\mathcal{S}$, then it is related to all vectors of weight at least $m$ in $\C_k^0$. Assume by contradiction that it is unrelated to a vector $q$ in $\C_k^0$, of weight at least $m$.
		
		Consider the transformation of $v$ into $q$ by changing some 1's to 0's and some 0's to 1's. Let $l'$ be the index of the first 1 that we change to 0. Then $l'\leq l$ (in the transformation given above, it was $l$, the index of the last 1). We can only change those bits from 0's to 1's which are before $l'$ (the number of 0's before $l'$ is at most $j$ since $l'\leq l$), or those which are between $l'+1$ and $l'+k-1$ (at most $k-1$); this is because the new vector will have a 0 at index $l'$ and so it cannot have 1's after index $l'+k-1$ if it is in $\C_k^0$. So if $i+1>j+k-1$, there are not enough 0's in $v$ which could be changed to 1's, so we cannot obtain a vector of weight $m$ or greater, which is in $\C_k^0$ and is unrelated to it.
	\end{proof}
	
	The following detail will be used in the proof of the next lemma.
	
	\begin{proposition}\label{worstset}
		Assume that a set $A$ in $\C_k^0$ is unrelated to some set of size $m$, but it is related
		to all sets of size $m+1$. Then there is a unique set of size $m$ in $\C_k^0$ unrelated
		to it: the one with an indicator vector
		$\overbrace{111\ldots1}^{m-k+1}0\overbrace{11\ldots1}^{k-1}\overbrace{000\ldots0}^{n-m-1}$.
	\end{proposition}
	\begin{proof}
		We have seen in the proof of Lemma \ref{unrelated} that among all sets in $\C_k^0$, a set $A$ of size $m-i$ is unrelated to some set of size $m$ if and only if $i+1\leq j+k-1$ where $j$ is the number of the number of 0's before the last 1 in the indicator vector of $A$. So $A$ (that is of size $(m+1)-(i+1)$) is related to all sets of size $m+1$ if and only if $(i+1)+1 > j+k-1$. By combining the two inequalities, we have $j+k-1=i+1$. Consider the transformation we have seen in the proof of Lemma \ref{unrelated}. The only way we can obtain an indicator vector of weight $m$ corresponding to $A$ is by removing the last 1 in its indicator vector, and changing all 0's before the last 1, plus the next $k-1$ after it, to 1's. Thus, the only set of size $m$ in $\C_k^0$ which is unrelated to $A$ is the one with an indicator vector $\overbrace{111\ldots1}^{m-k+1}0\overbrace{11\ldots1}^{k-1}\overbrace{000\ldots0}^{n-m-1}$.
	\end{proof}

	\begin{lemma}\label{maxsubset}
		For any poset $P$ of size $\abs{P}$ and height $h$, we have 
		\begin{displaymath}
			\alpha(\C_k,P) \le \abs P + (h-1)(3k-5)2^{k-2} - 1.
		\end{displaymath}
	\end{lemma}
	\begin{proof}
		We show that if $\h \subseteq \C_k^0$ with $\abs \h \geq \abs P + (h-1)(3k-5)2^{k-2}$, then $\h$ contains $P$ as a subposet. We may notice that a $k$-interval chain on $[n]$ is a subposet of the levels $3k-3$ to $n'-k+1$ of a $k$-interval chain on the larger base set $[n']$ where $(n'-k+1)-(3k-3) = n$ (i.e., $n' = n+4k-4$), with the injection $2^{[n]} \ni A \mapsto \{1,2,\ldots,3k-3\} \cup \{a+3k-3 : a\in A\} \in 2^{[n']}$. So we can assume that the elements of $P$ are embedded from levels $3k-3$ to $n-k+1$ of the interval chain.
		
		Let $A, B \in \h$ be arbitrary sets. We define a total order $\prec_{\h}$ on $\h$: If $\abs A < \abs B$, then let $B \prec_{\h} A$. If $\abs A = \abs B = m$, then their order is chosen arbitrarily, except if one of them, say $B$, is the set with the indicator vector $\overbrace{111\ldots1}^{m-k+1}0\overbrace{11\ldots1}^{k-1}\overbrace{000\ldots0}^{n-m-1}$, then we let $A \prec_{\h} B$ (so $B$ is the largest w.r.t.\ $\prec_{\h}$ among the sets of size $m$ in $\h$).
		
		Mirsky's theorem \cite{mirsky} states that the height of any poset equals the minimum number of antichains into which it can be partitioned. We decompose $P$ into antichains $\A_1,\A_2,\ldots\A_h$, where the elements in $\A_i$ are bigger than or unrelated to elements in $\A_j$ for any $i>j$ and then map the antichains $\A_h, \A_{h-1}, \ldots, \A_1$ into $\h$ one after another, in this order, in $h$ steps as follows. First, we map the elements of $\A_h$ to the smallest $\abs{\A_h}$ sets of $\h$ with respect to the total order $\prec_{\h}$ we defined. The family of these elements of $\h$ is denoted $\h_h$. We then remove all sets in $\h$ which are not proper subsets of every set in $\h_h$.  The family of these removed sets is denoted $\I_h$; in other words, $\I_h$ is the family of sets in $\h$ which are not properly contained in at least one set of $\h_h$. (Notice that $\h_h \subseteq \I_h$.) 
		Now we map $\A_{h-1}$ to the smallest (w.r.t.\ $\prec_{\h}$) $\abs{\A_{h-1}}$ sets of $\h\setminus\I_h$, denoted $\h_{h-1}$. We proceed similarly: we denote the family of the sets in $\h$ which are not properly contained in every set of $\h_h\cup\ldots\cup\h_i$ with $\I_i$, and map $\A_{i-1}$ to the collection of smallest $\abs{\A_{i-1}}$ sets (w.r.t.\ $\prec_{\h}$) of $\h\setminus\I_i$, denoted $\h_{i-1}$. By this process, each set in $\h_i$ contains all the sets in $\h_j$ for $i>j$.
		
		We have to show that the process finishes before $\h$ is exhausted, that is, 
		\begin{equation}
			\label{toshow}
			\abs{\bigcup_{i=1}^h \h_i \cup \bigcup_{i=2}^h \I_i}\leq \abs P + (h-1)(3k-5)2^{k-2}.
		\end{equation}
		For this purpose, we show that for each $i\in\{h,h-1,\ldots,2\}$, the number of new sets that are removed at this step, besides $\h_i$: $\abs{\I_i\setminus\left(\h_i\cup\I_{i+1}\right)}$ is at most $(3k-5)2^{k-2}$ (where we consider $\I_{h+1}=\emptyset$). Since $\abs{\bigcup_{i=1}^h \h_i} = \abs{P}$ and there are $h(P)-1$ steps in which sets are removed, we will have our desired inequality \eqref{toshow}. Let $A$ be the largest set (w.r.t.\ $\prec_{\h}$) in $\h_i$, and $m=\abs A$. Every set which is smaller than $A$ (w,r,t $\prec_{\h}$) is either in $\h_i$ or $\I_{i+1}$. If $A=\overbrace{111\ldots1}^{m-k+1}0\overbrace{11\ldots1}^{k-1}\overbrace{000\ldots0}^{n-m-1}$, then $\I_i\setminus\left(\h_i\cup\I_{i+1}\right)$ is a subcollection of all sets in $\C_k^0$ whose size is smaller than $m$, but which are unrelated to at least one set in $\C_k^0$ of size $m$ or more. By Lemma \ref{unrelated}, the number of such sets is $(3k-5)2^{k-2}$. If $A\neq\overbrace{111\ldots1}^{m-k+1}0\overbrace{11\ldots1}^{k-1}\overbrace{000\ldots0}^{n-m-1}$, then, by Proposition~\ref{worstset}, the sets in $\C_k^0$ whose size is smaller than $m$, and which are unrelated to $A$ or some other set in $\h$ which is smaller than $A$ in our order $\prec_{\h}$, are also unrelated to some set in $\C_k^0$ of size $m+1$ or more. Thus the sets in $\I_i\setminus\left(\h_i\cup\I_{i+1}\right)$ are some sets in $\C_k^0$ of size $m$ and some sets whose size is smaller than $m$ but unrelated to at least one set in $\C_k^0$ of size $m+1$ or more. Again, the number of such sets is at most $(3k-5)2^{k-2}$.
	\end{proof}
	Now we are ready to prove our main result, Theorem \ref{mainthm}.
	\begin{proof} [Proof of Theorem \ref{mainthm}]
		Let $\A$ be a $P$-free family over $[n]$. Let $N_{\abs{A}}$ denote the number of sets of size $\abs A$ from the $k$-Interval chain. 
		
		\begin{align*}
			2^{k-1} \abs{\A} &= \sum_{\substack {A \in \A \\  \abs{A} < k \text{ or } \abs{A} > n-k }} 2^{k-1} + \sum_{\substack {A \in \A \\ k \le \abs{A} \le n-k }} 2^{k-1} \\ &\le
			\sum_{\substack {A \in \A \\  \abs{A} < k \text{ or } \abs{A} > n-k }} \frac{N_{\abs{A}} {n \choose \left\lfloor \frac{n}{2}\right\rfloor }}{\binom{n}{\abs{A}}} + \sum_{\substack {A \in \A \\ k \le \abs{A} \le n-k }} \frac{2^{k-1} \cdot {n \choose \left\lfloor \frac{n}{2}\right\rfloor }}{\binom{n}{\abs{A}}}\le \alpha(\C_k,P) {n \choose \left\lfloor \frac{n}{2}\right\rfloor } .
		\end{align*}
		
		If $\abs{A} < k \text{ or } \abs{A} > n-k$, we have $2^{k-1}  \le \frac{{n \choose \left\lfloor \frac{n}{2}\right\rfloor }}{\binom{n}{\abs{A}}}$ when $n$ is sufficiently large and so the first inequality holds. If $k \le \abs{A} \le n-k$, by Lemma \ref{levelsize}, we have $ 2^{k-1} = N_{\abs{A}}$ and so the second inequality holds due to Lemma \ref{bound}. Now we use Lemma \ref{maxsubset} to upper bound $\alpha(\C_k,P)$, from which the theorem follows.
	\end{proof}
	
	We now obtain Corollary \ref{maincor} using the above theorem.
	
	\begin{proof}[First proof of Corollary \ref{maincor}]
		Let $\A$ be a $P$-free family, and let $h$ be the height of $P$. Define $k=\left\lceil \log_{2}\left(\frac{\abs P}{h}\right)\right\rceil =\log_{2}\left(\frac{\abs P}{h}\right)+x=\log_{2}\left(\frac{\abs P y}{h}\right)$. Let us substitute this $k$ into Theorem \ref{mainthm} (where $0\leq x <1$ and $1\leq y <2$). 
		If $k\geq 2$, we get
		\begin{multline*}
			\frac{\abs{\A}}{{n \choose \left\lfloor \frac{n}{2}\right\rfloor }}
			\leq  \frac{1}{2^{k-1}} \left(\abs P + (h-1)(3k-5)2^{k-2} - 1 \right)
			<\frac{3\cdot2^{k-2}kh+\abs P}{2^{k-1}}=\\
			=\frac{\frac{3}{4}y\abs P\left(\log_{2}\left(\frac{\abs P}{h}\right)+x\right)+\abs P}{\frac{y\abs P}{2h}}<\frac{3}{2}\log_{2}\left(\frac{\abs P}{h}\right)h+3.5h.
		\end{multline*}
		If $k\leq 1$, we have $\abs P \leq 2h$. Double counting with just the chain gives a bound of $(\abs {P}-1) {n \choose \left\lfloor \frac{n}{2}\right\rfloor }$ (see Erd\H{o}s \cite{erdos1945lemma}), so the corollary still holds.
		So we have,
		\begin{displaymath}
			\La(n,P) < \left( \frac{3}{2}\log_{2}\left(\frac{\abs P}{h}\right)h+3.5h \right) {n \choose \left\lfloor \frac{n}{2}\right\rfloor }. \qedhere
		\end{displaymath}
	\end{proof}
	
	\section{A different proof of Corollary \ref{maincor} using generalized diamonds}
	We begin by recalling some results from the papers of Griggs and Li \cite{griggs2012poset} and Griggs, Li and Lu \cite{griggs2012diamond}.
	\begin{definition}[Product of posets]
		If a poset $P$ has a unique maximal element and a poset $Q$ has a unique minimal element, then their product $P \otimes Q$ is defined as the poset formed by identifying the maximal element of $P$ with the minimal element of $Q$.
	\end{definition}

	\begin{lemma}[Griggs, Li \cite{griggs2012poset}]
		\label{sumofposets}
		$\La(n, P \otimes Q) \le \La(n,P)+ \La(n, Q)$.
	\end{lemma}
	
	\begin{proof}
		Let $\F$ be a maximal $P \otimes Q$-free family. Define $\F_1 = \{S \in \F \mid \F \cap [S, [n]] \text{ contains } Q \}$ and let $\F_2 = \F \setminus \F_1$. 
		
		We claim that $\F_1$ is $P$-free. Suppose not. Then there is a set $M_1 \in \F_1$ which represents the maximal element of $P$, and, by definition, $\F \cap [M_1, [n]]$ contains $Q$. Also notice that, since $M_1$ represents the maximal element of $P$, there are no elements in $[M_1, [n]] \setminus \{M_1\}$ that are part of the representation of $P$. This implies that $\F$ contains $P \otimes Q$, a contradiction. It is easy to see that $\F_2$ is $Q$-free, for otherwise, the element $M_2$, that represents the minimal element of $Q$ satisfies: $\F \cap [M_2, [n]]$ contains $Q$, contradicting the definition of $\F_2$.  So we have $ \abs{\F} = \La(n, P \otimes Q) = \abs{\F_1} + \abs{\F_2} \le \La(n, P) + \La(n, Q)$, as desired.
	\end{proof}

	We shall write $h$ in place of $h(P)$ for convenience.
	Let $D_k$ be the poset on $k+2$ elements with relations $ b < c_1, c_2, \ldots, c_k < d$. Let $K_{a_1,\ldots,a_h}$ be the complete $h$-level poset where the sizes of levels are $a_1, a_2, \ldots, a_h$: the poset in which every element is smaller than every element on every higher level.
	
	By using a partition method on chains, Griggs, Li and Lu proved 
	
	\begin{theorem}[Griggs, Li, Lu \cite{griggs2012diamond}]
		\label{D_k}
		Let $k \ge 2$. Then,
		$$\La(n, D_k) \le (\log_2(k+2) + 2) \binom{n}{\floor{\frac{n}{2}}}.$$
	\end{theorem}
	
	\medskip
	
	By Mirsky's decomposition \cite{mirsky}, $P$ can be viewed as a union of $h$ antichains: $\A_i$, $1 \le i \le h$. Let $\abs{\A_i} = a_i$. Then, it is easy to see that the following lemma holds.

	\begin{lemma}
		\label{embedding}
		$P$ is a subposet of $K_{a_1,\ldots,a_h}$, which in turn, is a subposet of \\ $D_{a_1} \otimes D_{a_2} \otimes \ldots \otimes D_{h-1} \otimes D_{a_h}$.
	\end{lemma}
	
	Now we are ready to prove Corollary \ref{maincor} with better constants.
	
	%
	
	\begin{proof}[Second proof of Corollary \ref{maincor}]
		By Lemma \ref{embedding}, we have
		\[\La(n, P) \le \La(n, K_{a_1,\ldots,a_h}) \le \La(n, D_{a_1} \otimes D_{a_2} \otimes \ldots \otimes D_{a_{h-1}} \otimes D_{a_h}). \]
		By Lemma \ref{sumofposets} and Theorem \ref{D_k}, we have
		\[\La(n, D_{a_1} \otimes D_{a_2} \otimes \ldots \otimes D_{a_{h-1}} \otimes D_{a_h}) \le  \sum_{i=1}^{h} (\log_2(a_i+2) + 2) \binom{n}{\floor{\frac{n}{2}}}. \]
		Bounding the sum on the right-hand side, by Jensen's inequality we have
		\[
		\sum_{i=1}^{h} (\log_2(a_i+2) +2) \le h \cdot \log_2 \left( \frac{\abs{P}}{h} + 2 \right) + 2h.
		\]
		This implies our desired result
		\[ \La(n,P) \leq  \left( h \cdot \log_2 \left( \frac{\abs{P}}{h} + 2 \right) + 2 h \right) {n \choose \left\lfloor \frac{n}{2}\right\rfloor }. \qedhere \]
	\end{proof}
	
	Finally, we will prove Proposition \ref{lowerbound}, a matching lower bound for Corollary \ref{maincor}.
	
	\begin{proof} [Proof of Proposition \ref{lowerbound}]
		We show that the height of any poset corresponding to a family of sets which realizes $K_{a,a,\ldots,a}$ is at least $(h-2) \log_2 a +1$.
		This implies that if $\A$ is the middle $(h-2)\log_2 a$ levels of $2^{[n]}$, it does not contain $P$ as a subposet.
		
		Let us denote the levels of $P=K_{a,a,\ldots,a}$ by $\P_1,\P_2,\ldots,\P_h$, and let $\mathcal{H}$ be a set family into which $P$ is embedded. For every $1\leq i\leq h-1$, let $U_i$ be the union of the sets corresponding to the elements of $\P_i$ by the embedding. Then, the structure of $P$ implies that every element of $\P_{i+1}$ is mapped to sets containing $U_i$. If $\abs{U_{i+1} \setminus U_i}=k$, there are $2^k$ sets in total containing $U_i$ and contained in $U_{i+1}$. Thus, we have $\abs{U_{i+1}}-\abs{U_{i}} \geq \log_2 a$ (this idea comes from Theorem 2.5 in \cite{griggs2012diamond}). So $\abs{U_{h-1}}-\abs{U_{1}} \geq (h-2)\log_2 a$. $\P_1$ is mapped to sets of size at most $\abs{U_1}$, and $\P_h$ is mapped to sets of size at least $\abs{U_{h-1}}$, so the set family spans at least $(h-2)\log_2 a +1$ levels.
	\end{proof}
	
	\section{Proof of Theorem \ref{induced}}
	In this section we will give an upper bound on the size of the Lubell function of an induced $P$-free family.  Lemma \ref{bound} holds for induced posets as well by an identical proof. Let $0\leq a \leq b \leq n$. Let $\h\subset 2^{[n]}$ be a collection of sets which has the same number of sets, $N$, for each cardinality $i$ for $a \leq i \leq b$.  Define $\alpha^\#(\h,P)$ to be the size of the largest subcollection of $\h$ containing no induced $P$.
	\begin{lemma}\label{intervalbound}
		Let $\A$ be an induced $P$-free family in $2^{[n]}$, in which the cardinality of every set is between $a$ and $b$. We have
		\[l_n(\A) \leq \frac{\alpha^\#(\h,P)}{N}.\]
		In particular, if $\C_k$ is an interval chain as defined in the Section \ref{intchainsec}, and $k\leq a$ and $b\leq n-k$ hold, we have
		\[l_n(\A) \leq \frac{\alpha^\#(\{A\in\C_k : a\leq \abs A \leq b\},P)}{2^{k-1}}.\]
	\end{lemma}
	\begin{proof}
		The proof of Lemma \ref{bound} applies, observing that $a \leq \abs A \leq b$.
	\end{proof}
	
	We prove the following statement, which is slightly stronger than Theorem \ref{induced}.
	\begin{lemma}\label{inducedlemma}
		Let $P$ be a poset and let $c>\frac 1 2$. Let $n$ be a natural number, and let $0\leq a \leq b \leq n$. If $\A$ is an induced $P$-free family in which the cardinality of every set is between $a$ and $b$,
		\[l_n(\A)=\Ordo\left((b-a)^c\right).\]
	\end{lemma}
	
	The following claim will be used recursively and is key to the proof of our lemma.
	
	\begin{claim}
		\label{recursion}
		If Lemma \ref{inducedlemma} holds for a given $c=c'>\frac 1 2$, then it also holds for $c=\frac{2c'}{2c'+1}$.
	\end{claim}
	
	\begin{proof}[Proof of Claim]
		Let $m=b-a+1$, and let $k=m^{\frac{2}{2c'+1}}$. Let $\h=\{A\in\C_k : a+k\leq \abs A \leq b-k\}$. By definition $\C_k = \bigcup_{i=0}^{n-k} [A_i, A_{i+k}]$ (where $A_0\subset A_1 \subset \ldots \subset A_n$ is an arbitrary maximal chain), and the levels $a+k$ to $b-k$ intersect $m-k$ of the intervals $[A_i, A_{i+k}]$. By substituting $k$ in the place of $n$ in Theorem \ref{inducedLa}, there is a constant $C$ such that $\abs{\A\cap[A_i, A_{i+k}]}\leq C\binom{k}{\left\lfloor\frac k 2\right\rfloor}$ for every $i$. Thus $\alpha^\#(\h,P)\leq (m-k)C\binom{k}{\left\lfloor\frac k 2\right\rfloor}<Cm\binom{k}{\left\lfloor\frac k 2\right\rfloor}$. By Lemma \ref{intervalbound},
		
		\begin{equation}
			\label{A}
			l_n(\{A\in\C_k : a+k\leq \abs A \leq b-k\})\leq Cm\frac{\binom{k}{\left\lfloor\frac k 2\right\rfloor}}{2^{k-1}}\leq \frac{2\sqrt{2}}{\sqrt{\pi}}C\frac{m}{\sqrt k}=\frac{2\sqrt{2}}{\sqrt{\pi}}C\frac{m}{\sqrt{m^{\frac{2}{2c'+1}}}}=\frac{2\sqrt{2}}{\sqrt{\pi}}Cm^{\frac{2c'}{2c'+1}}.
		\end{equation}
		By our assumption, using Lemma \ref{inducedlemma} with substituting $a+k-1$ in the place of $b$, we have
		\begin{equation}
			\label{B}
			l_n(\{A\in\C_k : a\leq \abs A \leq a+k-1\})=\Ordo\left(k^{c'}\right)=\Ordo\left(m^{\frac{2c'}{2c'+1}}\right).
		\end{equation}
		Similarly, by substituting $b-k+1$ in the place of $a$, we have
		\begin{equation}
			\label{C}
			l_n(\{A\in\C_k : b-k+1\leq \abs A \leq b\})=\Ordo\left(m^{\frac{2c'}{2c'+1}}\right).
		\end{equation}
		Adding up the inequalities \eqref{A}, \eqref{B} and \eqref{C}, we get 
		
		\[l_n(\{A\in\C_k : a\leq \abs A \leq b\})=\frac{2\sqrt{2}}{\sqrt{\pi}}Cm^{\frac{2c'}{2c'+1}}+2\Ordo\left(m^{\frac{2c'}{2c'+1}}\right)=\Ordo\left((b-a)^{\frac{2c'}{2c'+1}}\right). \qedhere\]
	\end{proof}
	
	\begin{proof}[Proof of Lemma \ref{inducedlemma}]
		The lemma is trivial for $c=1$. Substituting $c=1$ in the proof of the claim directly gives a proof for $c=\frac 2 3$. Then, applying the claim recursively proves the statement for a sequence of exponents $c=c_i=\frac{2^i}{2^{i+1}-1}$. Indeed, \[\frac{2c_i}{2c_i+1}=\frac{2\frac{2^i}{2^{i+1}-1}}{2\frac{2^i}{2^{i+1}-1}+1}=\frac{2^{i+1}}{2^{i+2}-1}=c_{i+1}.\]
		The limit of the sequence is $\frac 1 2$, so it eventually becomes smaller than any $c>\frac 1 2$, proving our lemma.
	\end{proof}
	
	\section{Acknowledgements}
	We would like to thank Wei-Tian Li for many helpful discussions during the SUM(M)IT 240 conference in Budapest.  We also wish to thank the anonymous referees and D\"om\"ot\"or P\'alv\"olgyi for their careful reading of our manuscript and helpful remarks.
	
	\bibliography{bibliography1}  

\end{document}